\documentclass{article}
\usepackage{natbib}
\usepackage{array}
\usepackage{mathtools}
\usepackage{amssymb}
\usepackage{mathabx}
\usepackage[amsmath, amsthm, thmmarks]{ntheorem}
\usepackage{bbm}
\usepackage{graphicx} 
\usepackage{placeins}
\usepackage{longtable}
\usepackage{epsfig}
\usepackage{appendix}
\usepackage{xcolor}
\usepackage{algorithm}
\usepackage{algpseudocode}
\usepackage{stmaryrd}
\usepackage{gauss}
\usepackage{geometry}
\usepackage{pdflscape}
\newcommand{\acts}{\mathrel{\reflectbox{$\righttoleftarrow$}}}
\newcommand{\g}{{\mathbf{g}}}

\newcommand{\Db}{{\mathcal{D}}}
\newcommand{\E}{\mathbbm{E}}

\newcommand{\fns}{\footnotesize}

\newtheorem{lemma}{Lemma}[section]
\newtheorem{theorem}{Theorem}[section]
\newtheorem{corollary}[theorem]{Corollary}

\title{Combinatorics of a dissimilarity measure for pairs \\ of draws from discrete probability vectors \\ on finite sets of objects}
\author{Zarif Ahsan\thanks{Department of Biology, Stanford University, Stanford, CA 94305 USA} , Xiran Liu$^*$, Noah A.~Rosenberg$^*$}

\begin{document}
\maketitle

\begin{abstract}
\noindent  Motivated by a problem in population genetics, we examine the combinatorics of dissimilarity for pairs of random unordered draws of multiple objects, with replacement, from a collection of distinct objects. Consider two draws of size $K$ taken with replacement from a set of $I$ objects, where the two draws represent samples from potentially distinct probability distributions over the set of $I$ objects. We define the set of \emph{identity states} for pairs of draws via a series of actions by permutation groups, describing the enumeration of all such states for a given $K \geq 2$ and $I \geq 2$. Given two probability vectors for the $I$ objects, we compute the probability of each identity state. From the set of all such probabilities, we obtain the expectation for a dissimilarity measure, finding that it has a simple form that generalizes a result previously obtained for the case of $K=2$. We determine when the expected dissimilarity between two draws from the same probability distribution exceeds that of two draws taken from different probability distributions. We interpret the results in the setting of the genetics of polyploid organisms, those whose genetic material contains many copies of the genome ($K > 2$).
\end{abstract}

\section{Introduction}

\subsection{Biological motivation}

In many species including humans, the complete genetic material of an individual contains two copies of the genome, one inherited from each of two parents. For a specific location in the genome, a \emph{locus}, many variant types, or \emph{alleles}, might exist. At a locus, the \emph{genotype} of an individual is the unordered pair of its alleles, a multiset containing two elements. For example, an individual inheriting allele $A$ from one parent and allele $B$ from the other has genotype $AB$ (equivalently, $BA$). 

Humans are \emph{diploid}: each individual has two copies of the genome. Some organisms are \emph{haploid}, with only one copy. Others, including perhaps half of plant species~\citep{HeslopHarrison2023}, are \emph{polyploid}, possessing 4, 6, 8, 10, 12, or as many as 96 genomic copies~\citep{Nagalingum2016}; this polyploidy is the outcome of past genome duplication and hybridization events. The \emph{ploidy} of an organism is its number of copies of the genome; in an organism with ploidy $K$, a genotype is a multiset with $K$ elements.

Population-genetic studies often seek to understand features of the alleles in a population of individuals and the differences in alleles and their frequencies among two or more populations. Dissimilarity scores between genotypes or populations of genotypes are often used for such computations~\citep{ChakrabortyAndJin1993, MountainAndRamakrishnan2005, Rosenberg2011, Liu2023}. Given two diploid genotypes sampled from probability distributions on a set of alleles, what is the distribution of a dissimilarity score between the genotypes? The dissimilarity for two genotypes is a discrete value that depends on the number of identical alleles between the genotypes, in a manner we define in Section \ref{sec:notation}. If allele frequencies are taken into account to weight the discrete values by their probabilities, then the dissimilarity ranges continuously over $[0,1]$. The diploid genotypes can be samples from the same probability distribution, for two draws from the same population of individuals, or from different distributions, for draws from different populations.

The mathematical study of population-genetic dissimilarity statistics has focused on diploid genotypes. For example, for two dissimilarity statistics, \cite{Liu2023} obtained the expected dissimilarity for two diploid genotypes sampled from the same population and for two genotypes sampled from different populations (eq.~22). In other words, for two unordered draws of size two, with replacement, from two specified distributions,  \cite{Liu2023} obtained the expectation of the dissimilarity. Such computations help to understand how allele frequency distributions give rise to the dissimilarity scores used in population-genetic data analysis.

The mathematical statistics of polyploid organisms has been explored to a lesser extent than that of diploids, and our interest here is in generalizing computations of population-genetic dissimilarity to arbitrary ploidy. We seek to obtain the expected value of the dissimilarity for two genotypes of specified ploidy sampled from the same population and for two genotypes sampled from different populations, In other words, for two unordered draws of size $K$, with replacement, from two specified probability distributions, what is the expectation of the dissimilarity? 

\subsection{Problem description}

The problem of interest can be understood beyond the population genetics context in a general combinatorial setting. Consider two random, unordered draws of $K$ items with replacement from a set of $I$ objects, where each of the $I$ objects is associated with a probability. The two draws potentially come from populations with different frequencies for the $I$ objects. The $i$th object has probability $p_i$ in one population and $q_i$ in the other, where $p_i \geq 0$, $q_i \geq 0$, and $\sum_{i=1}^I p_i=\sum_{i=1}^I q_i=1$. In the context of genetics, $K$ is the ploidy, $I$ is the number of distinct alleles, and $p_i$ and $q_i$ are allele frequencies in populations 1 and 2, respectively.

We seek to solve three problems. First, we find the number of distinct \emph{identity states}, up to relabeling of the alleles, where the identity states correspond to equivalence classes of pairs of unordered draws. Next, for each identity state, we compute its probability and its associated dissimilarity. Finally, we find the expected dissimilarity between two individuals selected at random, one from one probability distribution and one from another. We consider the general case in which the two draws are taken from different probability distributions; we also consider the special case in which the probability distributions are the same. In the genetic context, the former case corresponds to comparing genotypes from different populations and the latter  to comparing genotypes from the same population.

\subsection{Notation}
\label{sec:notation}

Let $\mathcal{A}_I$ be a set of $I$ distinct objects, which we label $A_i$ for $i\in\{1,2,\ldots,I\}$. The Cartesian product $\mathcal{A}_I^K$ describes an \emph{ordered} draw of $K$ elements from $\mathcal{A}_I$, with replacement. We have a group action $S_K\acts \mathcal{A}_I^K$ given by permuting the order of elements for each $(A_{i_1}, A_{i_2}, \ldots, A_{i_K})\in\mathcal{A}_I^K$. We can then define the space of \emph{unordered} draws of size $K$ from $\mathcal{A}_I$, denoted by $\mathcal{G}_I^K$, as a quotient by this group action, that is, $\mathcal{G}_I^K = \mathcal{A}_I^K/S_K$. 

In general, we indicate ordered draws with the letter $X$ and unordered draws with the letter $G$. Two ordered draws $X_1=(X_1^1, X_1^2, \ldots, X_1^K)$ and $X_2=(X_2^1, X_2^1, \ldots, X_2^K)$  correspond to the same unordered draw
if they are contained in the same orbit by the $S_K$ action. Each class $G\in\mathcal{G}^K_I$ can be uniquely represented by a vector $\g=(g^{(1)}, g^{(2)}, \ldots, g^{(I)})$, where $g^{(i)}$ is the count of object $A_i$ in $G$ and $\sum_{i=1}^I g^{(i)} = K$. 

In our population-genetic example, $\mathcal{A}_I$ denotes the set of alleles present in at least one of two populations. $A_i$ is the label for the $i$th allele, and $p_i$ and $q_i$ are the frequencies of $A_i$ in the two populations. Each $G\in\mathcal{G}^K_I$ is a possible genotype. 

We define a dissimilarity $\Db:\mathcal{G}^K_I \times \mathcal{G}^K_I \rightarrow [0,1]$ for pairs of unordered draws. For draws $G_1$ and $G_2$,
\begin{align}
\label{eq:dg1g2}
    \Db(G_1,G_2)=1-\frac{1}{K^2}\sum_{1\leq i, j \leq K} \mathbbm{1}_{G_1^{(i)} = G_2^{(j)}}.
\end{align}
Equivalently, with vector representations $\g_1$ and $\g_2$ for $G_1$ and $G_2$, 
\begin{align}
    \Db(\g_1,\g_2)=1-\frac{1}{K^2}\langle \g_1, \g_2 \rangle,
\end{align}
where $\langle \cdot, \cdot \rangle$ is the standard inner product. 

Note that $\Db$ is not a distance metric, as for any $K \geq 2$ and $I \geq 2$, for the vector $\g=(c,K-c,0,\ldots,0)$ with positive $c$, we have $\Db(\g,\g)=1-\frac{1}{K^2}[c^2+(K-c)^2] \neq 0$. Note also that although it is convenient to consider the range for $\Db$ to be the full unit interval, the range is only the set $\{0, {1}/{K^2}, {2}/{K^2}, \ldots, (K^2-1)/{K^2}, 1\}$.

Across all possible $\g_1,\g_2$, the minimum value of $\Db=0$ is reached if and only if all objects have the same kind: $\g_1 = \g_2 = K \mathbf{e}_i$, where $\mathbf{e}_i$ is the vector with 1 in the $i$th position and 0 in the remaining $I-1$ positions. The maximum value of $\Db=1$ is reached if and only if the two draws share no common items: that is, if and only if there is no value of $i$ for which $\g_1$ and $\g_2$ both have nonzero values.

$\Db$ generalizes the diploid dissimilarity termed $D_2$ by \cite{Liu2023}, reducing to $D_2$ if $K=2$. Example draws with $K=2$ and $I=4$ appear in Table~\ref{tab:example-draws}. Notation appears in Table \ref{tab:notation}.  
\begin{table}[tb]
\caption{Pairs of example draws $(G_1,G_2)$, their vector representations $(\g_1,\g_2)$, and the dissimilarities ($\Db$) between them for $K=2$ and $I=4$.}\label{tab:example-draws}
\vspace{0.2cm}
    \centering
    \begin{tabular}{|c|c|c|c|c|}
    \hline
        $G_1$       & $G_2$       & $\g_1$      & $\g_2$      & $\Db$ \\[0.8ex] \hline
        $A_1 A_1$ & $A_1 A_1$ & $(2,0,0,0)$ & $(2,0,0,0)$ & 0     \\[0.8ex]
        $A_1 A_1$ & $A_1 A_2$ & $(2,0,0,0)$ & $(1,1,0,0)$ & $\frac{1}{2}$ \\[0.8ex] 
        $A_1 A_1$ & $A_2 A_2$ & $(2,0,0,0)$ & $(0,2,0,0)$ & 1     \\[0.8ex]
        $A_1 A_1$ & $A_2 A_3$ & $(2,0,0,0)$ & $(0,1,1,0)$ & 1     \\[0.8ex]
        $A_1 A_2$ & $A_1 A_2$ & $(1,1,0,0)$ & $(1,1,0,0)$ & $\frac{1}{2}$ \\[0.8ex]
        $A_1 A_2$ & $A_1 A_3$ & $(1,1,0,0)$ & $(1,0,1,0)$ & $\frac{3}{4}$ \\[0.8ex]   
        $A_1 A_2$ & $A_3 A_4$ & $(1,1,0,0)$ & $(0,0,1,1)$ & 1     \\[0.8ex] \hline
    \end{tabular}
\end{table}

\begin{table}[tb]
\caption{Notation table.}\label{tab:notation}
\vspace{.2cm}
    \centering
    \begin{tabular}{|l|p{5.4in}|}
    \hline
Symbol & Description  \\ \hline
$K$    & Number of items drawn with replacement from a set of objects (ploidy) \\
$I$    & Number of distinct objects that can be sampled (number of distinct alleles) \\ 
$p_i$  & Probability that object $i$ is drawn in population 1 (frequency of allele $i$ in population 1) \\
$q_i$  & Probability that object $i$ is drawn in population 2 (frequency of allele $i$ in population 2) \\ 
$A_i$  & $i$th object in a set of $I$ distinct objects (allelic type $i$) \\
$\mathcal{A}_I$   & Set of $I$ distinct objects (set of distinct allelic types) \\
$\mathcal{A}_I^K$ & Set of ordered draws of size $K$ from a set of $I$ distinct objects \\
$S_K$  & Group action that permutes the order of elements in vectors belonging to $\mathcal{A}_I^K$ \\
$\mathcal{G}_I^K$ & Set of unordered draws of size $K$ \\
$\mathcal{C}_I^K$ & Set of unordered pairs of unordered draws of size $K$, after permutation of the object labels \\
$\mathcal{D}$ & Dissimilarity computed between pairs of unordered size-$K$ draws from a set of $I$ objects \\
$X$  & Ordered draw of size $K$ from a set of $I$ distinct objects \\
$G$  & Unordered draw of size $K$ from a set of $I$ distinct objects \\
$\g$ & Vector that counts in a draw of size $K$ the numbers of copies of the $I$ elements 
\\
$\mathbf{r}$ & Unordered list of nonzero entries in a vector \\
$\hat{\g}$        & Ordered pair of unordered size-$K$ draws from a set of $I$ objects \\
$M$               & $(K+1)\times (K+1)$ matrix representation of $\hat{\g}$ \\
$\mathcal{M}$     & set of all $(K+1)\times (K+1)$ matrix representations in $\mathcal{C}_I^K$\\
$[\hat{\g}]_\sim$ & Unordered pair of unordered size-$K$ draws from a set of $I$ objects \\  
$N(\hat{\g})$     & Number of nonzero columns in a $2 \times I$ matrix $\hat{\g}$ \\
$[\hat{\g}]$      & Identity state associated with an ordered pair of unordered draws $\hat{\g}$ \\  
\hline
\end{tabular}
\end{table}

\section{Enumeration of identity states}
\label{sec:enumeration}

Our first step toward studying probabilities of identity states is to enumerate the identity states, the equivalence classes for pairs of unordered draws of size $K$ with replacement from a set of $I$ objects, up to relabeling; it is convenient to define identity states for \emph{unordered} pairs of unordered draws. We define the possible identity states through a series of group actions. Recall that $\mathcal{G}^K_I$, the set of unordered draws of size $K$ from a set of $I$ objects, with replacement, is defined as $\mathcal{G}^K_I = \mathcal{A}^K_I/S_K$, where $\mathcal{A}^K_I$ is the set of ordered draws and $S_K$ is the group action that permutes the order of elements of $\mathcal{A}^K_I$.

$\mathcal{G}^K_I \times \mathcal{G}^K_I$ corresponds to \textit{ordered} pairs of unordered draws. We can uniquely represent an element $(G_1,G_2)\in\mathcal{G}^K_I \times \mathcal{G}^K_I$ by concatenating the vectors $\g_1$ and $\g_2$ to construct a $2\times I$ matrix, denoted  
\begin{equation}\label{eq:pair-matrix-form}
    \hat{\g}=\begin{pmatrix}
    \g_1 \\ \g_2
    \end{pmatrix} =
     \begin{pmatrix}
    g_1^{(1)}&g_1^{(2)}&\ldots&g_1^{(I)} \\ g_2^{(1)}&g_2^{(2)}&\ldots&g_2^{(I)}
    \end{pmatrix}.
\end{equation}
The dissimilarity function $\Db$ is symmetric, so that the order of the two draws can be reversed without changing its value. It is therefore convenient to define identity states in terms of an unordered pair of vectors rather than an ordered pair. Importantly, our notion of dissimilarity does not depend on the particular labeling of objects, that is, the assignment of the labels $A_1, A_2, \ldots, A_I$ to the $I$ objects. The dissimilarity between $A_1 A_1$ and $A_1 A_2$ would not be changed if we permuted labels to, say, $A_3 A_3$ and $A_3 A_5$. 

To account for the symmetries---due to the order of the two draws and to the $K$ elements in a draw---we define a group action $S_2 \times S_I \acts \mathcal{G}^K_I \times \mathcal{G}^K_I$. The $S_2$ corresponds to the symmetry of $\Db$ in its two arguments and the $S_I$ to the symmetry due to relabeling of the objects. Specifically, for a permutation $\tau\in S_I$, $\tau: m \mapsto n $ for $m,n\in\{1,2,\ldots,I\}$ if and only if $\tau \big( (G_1, G_2)\big)$ replaces all $A_m$ in $(G_1,G_2)$ with $A_n$.

As an example, consider $K=2$ and $I=4$. The element $\big( (12),1\big)\in S_2 \times S_I$ converts the pair of draws $(A_1 A_2, A_3 A_4)$ to $(A_3 A_4, A_1 A_2)$, switching the order of the draws. Element $\big(1,(23)\big)\in S_2 \times S_I$ converts $(A_1 A_2, A_3 A_4)$ to $(A_1 A_3, A_2 A_4)$, switching the labels $A_2$ and $A_3$. Element $\big( (12), (23) \big)\in S_2 \times S_I$ converts $(A_1 A_2, A_3 A_4)$ to $(A_2 A_4, A_1 A_3)$, switching the order of the draws and switching the labels of $A_2$ and $A_3$. 

Representing $(G_1,G_2)$ in the matrix form $\hat{\g}$, the group action permutes the rows and columns of the matrix. The first component permutes rows, and the second component permutes columns. Hence, considering the same example, we can represent $(A_1 A_2, A_3 A_4)$ as the matrix 
$\begin{psmallmatrix}
1 & 1 & 0 & 0 \\
0 & 0 & 1 & 1
\end{psmallmatrix}.$
The element $\big( (12),1 \big)$ takes this matrix to
$\begin{psmallmatrix}
0 & 0 & 1 & 1 \\
1 & 1 & 0 & 0
\end{psmallmatrix}.$
The element $\big( 1, (23) \big)$ takes it to
$\begin{psmallmatrix}
1 & 0 & 1 & 0 \\
0 & 1 & 0 & 1
\end{psmallmatrix}.$
The element $\big( (12), (23) \big)$ takes it to
$\begin{psmallmatrix}
0 & 1 & 0 & 1 \\
1 & 0 & 1 & 0 
\end{psmallmatrix}.$

The quotient of $\mathcal{G}^K_I \times \mathcal{G}^K_I$ by the group action $S_2 \times S_I$ corresponds to all unordered pairs of unordered draws up to relabeling. We denote this quotient by $\mathcal{C}^K_I=(\mathcal{G}^K_I \times \mathcal{G}^K_I) /(S_2 \times S_I)$. Note that although we combine the $S_2$ and $S_I$ symmetries as one group action here, it will also be useful to separate them into distinct actions $S_2\acts\mathcal{G}^K_I$ and $S_I \acts \mathcal{G}^K_I$. These actions are simply the induced actions of the two groups $S_2$ and $S_I$ as subgroups of $S_2 \times S_I$. That is, for $\sigma \in S_I$, $\sigma(\hat{\g}) = (1,\sigma) (\hat{\g})$ for $(1,\sigma) \in S_2 \times S_I $.

Each  $\hat{\g}\in\mathcal{G}^K_I \times \mathcal{G}^K_I$ has an associated element in  $\mathcal{C}^K_I$, its associated \emph{identity state}, which we denote $[\hat{\g}]$. Hence, the enumeration of identity states consists in counting distinct elements $[\hat{\g}]$ in $\mathcal{C}^K_I$: the number of $2\times I$ nonnegative integer matrices whose rows sum to $K$, up to permutation of rows and columns.

A lemma clarifies that the number of elements of $\mathcal{C}^K_I$ does not depend on $I$, provided $I \geq 2K$. The lemma encodes the idea that because in total, $2K$ items are contained in a pair of draws, at most $2K$ labels appear in those draws. 
\begin{lemma}
    For $I\geq2K$, $\big|\mathcal{C}^K_I \big| = \big|\mathcal{C}^K_{2K} \big|$.
\end{lemma}
\begin{proof}
Consider a matrix $\hat{\g}\in\mathcal{G}^K_I \times \mathcal{G}^K_I$ for $I\geq 2K$. Because $\hat{\g}$ has nonnegative entries and its rows $\g_1$ and $\g_2$ each sum to $K$, each row of $\hat{\g}$ has at most $K$ nonzero entries, and $\hat{\g}$ has at most $2K$ nonzero columns. 

By permutation of the columns of $\hat{\g}$, we obtain a representative $[\hat{\g}]$ whose last $I-2K$ columns each contain a pair of zeroes, and whose leftmost $2K$ columns give a $2\times 2K$ nonnegative-integer matrix with rows summing to $K$: an element of $\mathcal{G}^K_{2K}\times \mathcal{G}^K_{2K}$. We therefore have a bijection between $\mathcal{C}^K_I$ and $\mathcal{C}^K_{2K}$, namely
\begin{align*}
       [\hat{\mathbf{g}}] & \longmapsto 
       \big[ \begin{pmatrix} 
       \hat{\mathbf{g}} & 
0 \atop 0 & 0 \atop 0 & \cdots \atop \cdots & 0 \atop 0        
        \end{pmatrix} \big].
\end{align*}
The number of elements in $\mathcal{C}^K_I$ is therefore equal to the number of elements in $\mathcal{C}^K_{2K}$.
\end{proof}

We perform exhaustive enumeration to obtain $\big|\mathcal{C}^K_{2K}\big|$. Algorithmically, we proceed as follows. 
\begin{enumerate}
    \item  Enumerate potential entries for $\hat{\g}_1$: we first enumerate all unordered partitions $\mathbf{p}_1$ of $K$. For each unordered partition, we enumerate all ordered partitions $\mathbf{p}_1^\prime$. For each ordered partition, we enumerate all ${2K \choose |\mathbf{p}_1|}$ placements $\mathbf{p}_1^{\prime\prime}$ of the ordered sequence of $|\mathbf{p}_1|$ nonzero elements in a vector of length $2K$.

    \item  Enumerate potential entries for $\hat{\g}_2$: we next enumerate all unordered partitions $\mathbf{p}_2$ of $K$.. For each unordered partition, we enumerate all ordered partitions $\mathbf{p}_2^\prime$. For each ordered partition, we enumerate all ${2K \choose |\mathbf{p}_2|}$ placements $\mathbf{p}_2^{\prime\prime}$ of the ordered sequence of $|\mathbf{p}_2|$ nonzero elements in a vector of length $2K$.

    \item Construct a matrix that is invariant with respect to the permutation of the columns of the potential pair $\hat{\g}$: for each pair $(\mathbf{p}_1^{\prime\prime}, \mathbf{p}_2^{\prime\prime})$, we construct the $(K+1)\times(K+1)$ matrix $M(\mathbf{p}_1^{\prime\prime}, \mathbf{p}_2^{\prime\prime})$ with rows and columns indexed from 1 to $K+1$, where entry $M_{ij}$ is the number of indices $k=1,2,\ldots,2K$ for which $(\mathbf{p}_1^{\prime\prime})_k=i-1$ and $(\mathbf{p}_2^{\prime\prime})_k=j-1$.

    \item Tabulate identity states by accounting for exchanges of the rows of the potential pair $\hat{\g}$: we identify all unique matrices $M$ produced at the end of Step 3. If $M$ appears, then its transpose also appears, as a result of an exchange of the associated $\mathbf{p}_1^{\prime\prime}$ and $\mathbf{p}_2^{\prime\prime}$. The number of identity states is the number of symmetric matrices plus half the number of asymmetric matrices.
\end{enumerate}

Step 3 of this enumeration develops a nearly one-to-one way to represent each element $[\hat{\g}]\in\mathcal{C}^K_I$ as a $(K+1) \times (K+1)$ matrix $M$. Indeed, given a matrix $\hat{\g}$, consider the map $\hat{\g} \mapsto M$, where $M$ is a $(K+1)\times(K+1)$ matrix, with $M_{ij}$ indicating the count of the column $\begin{psmallmatrix} i-1 \\ j-1 \end{psmallmatrix}$ in $\hat{\g}$. The column order of $\hat{\g}$ does not matter for this map, so that the map descends to a map on the quotient $\mathcal{G}^K_I \times \mathcal{G}^K_I / S_I$. 
We seek to descend to a map $\mathcal{G}^K_I \times \mathcal{G}^K_I / (S_2\times S_I)$. This descent, however, is not immediate. Consider, for instance, $K=2$ and $I=4$. We have the mapping of matrices 
\begin{align*}
    \begin{pmatrix}
        2 & 0 & 0 & 0 \\
        1 & 1 & 0 & 0
    \end{pmatrix} &\mapsto
    \begin{pmatrix}
        2 & 1 & 0 \\
        0 & 0 & 0 \\
        0 & 1 & 0        
    \end{pmatrix}\\
    \begin{pmatrix}
        1 & 1 & 0 & 0 \\
        2 & 0 & 0 & 0
    \end{pmatrix} &\mapsto
    \begin{pmatrix}
        2 & 0 & 0 \\
        1 & 0 & 1 \\
        0 & 0 & 0        
    \end{pmatrix}.
\end{align*}
These two matrices are equivalent by our $S_2\times S_I$ action---in particular, the $S_2$ component action. However, it can be seen that $S_2$ acts on the image of these maps by transposing the $3 \times 3$ matrix. Thus, if we let $M$ be a $(K+1)\times(K+1)$ matrix \textit{up to a transpose}, then the map $\hat{\g} \mapsto M$ associates an element of $\mathcal{G}^K_I \times \mathcal{G}^K_I /( S_2\times S_I) = \mathcal{C}^K_I$ with a matrix. Formally, to obtain a bijection between identity states and matrices, we consider the image of the matrix mapping, which is the set
\begin{align*}
    \mathcal{M} = \bigg\{ M \in \text{Mat}_{(K+1)\times(K+1)}(\mathbb{Z}_{\geq 0}): \sum_{i=1}^{K+1} \sum_{j=1}^{K+1} M_{ij} = I \text{ and } \sum_{i=1}^{K+1} \sum_{j=1}^{K+1} (i-1) M_{ij} = \sum_{i=1}^{K+1} \sum_{j=1}^{K+1} (j-1) M_{ij} = K\bigg\}.
\end{align*}
In particular, we have a bijection between $\mathcal{C}^K_I$ and $\mathcal{M}/^T$, where $^T$ denotes action by transposing the matrices.

The computation time in the algorithm can be reduced by restricting attention in Step 1, without loss of generality, to a single ordering $\mathbf{p}_1^\prime$ for each unordered partition $\mathbf{p}_1$, and a single placement $\mathbf{p}_1^{\prime\prime}$ of its nonzero elements in a vector of length $2K$---say, with the entries of the unordered partition placed in nonincreasing order, starting in the first element of the vector, and with vector completed by zeroes.

The enumeration produces 2, 7, 21, 66, 192, and 565 identity states for $K=1$, 2, 3, 4, 5, and 6. $\big|\mathcal{C}^K_{2K}\big|$ follows A331722 in the On-Line Encyclopedia of Integer Sequences~\citep{oeis}. Table \ref{tab:3prob-combs-k2} displays the enumeration for the case of $K=2$, providing the 7 elements in $\mathcal{C}^2_4$. The table provides the $(K+1) \times (K+1)$ square matrix form for the entries of Table \ref{tab:example-draws}. Table \ref{tab:4prob-combs-k3} provides the 21 elements in $\mathcal{C}^3_6$. For $1 \leq I \leq 2K$,  $\big|\mathcal{C}^K_{I}\big|$ can be obtained from the $\big|\mathcal{C}^K_{2K}\big|$ identity states by counting states for which the number of nonzero columns in the $2 \times 2K$ matrix is less than or equal to $I$ (Table \ref{tab:5}).

\clearpage
\begin{longtable}[tb]{|l|c|c|c|c|c|l|}\caption{Enumeration of identity states in $\mathcal{C}^K_I$ for $K=2$, $I \geq 4$, each given by an associated pair of draws $[\hat{\g}]$ and a square matrix $M$; dissimilarities $\Db$ and probabilities are also shown.}\label{tab:3prob-combs-k2} \\
\hline
        Case & $G_1$ & $G_2$ & $[\hat{\g}]$ & $M$ & $\Db$&Probability \\\hline
        1 & $A_1A_1$ & $A_1A_1$ & $\fns{\begin{pmatrix} 2&0&0&0\\2&0&0&0 \end{pmatrix}}$ & $\fns{\begin{pmatrix} 3&0&0\\0&0&0 \\ 0&0&1 \end{pmatrix}}$ & $0$&$\sum_{i_1}^I {p_{i_1}^2 q_{i_1}^2 }$\\[0.8ex] \hline
        2 & $A_1A_1$ & $A_1A_2$ & $\fns{\begin{pmatrix} 2&0&0&0\\1&1&0&0 \end{pmatrix}}$ & $\fns{\begin{pmatrix} 2&1&0\\0&0&0 \\ 0&1&0 \end{pmatrix}}$ & $\frac{1}{2}$&$ 2 \sum_{i_1\neq i_2}^I {p_{i_1}^2 q_{i_1} q_{i_2} + p_{i_1} p_{i_2} q_{i_1}^2 }$\\[0.8ex] \hline
        3 & $A_1A_1$ & $A_2A_2$ & $\fns{\begin{pmatrix} 2&0&0&0\\0&2&0&0 \end{pmatrix}}$ & $\fns{\begin{pmatrix} 2&0&1\\0&0&0 \\ 1&0&0 \end{pmatrix}}$ & $1$&$\sum_{i_1\neq i_2}^I {p_{i_1}^2 q_{i_2}^2 }$\\[0.8ex] \hline
        4 & $A_1A_1$ & $A_2A_3$ & $\fns{\begin{pmatrix} 2&0&0&0\\0&1&1&0 \end{pmatrix}}$ & $\fns{\begin{pmatrix} 1&2&0\\0&0&0 \\ 1&0&0 \end{pmatrix}}$ & $1$&$\sum_{i_1\neq i_2\neq i_3}^I {p_{i_1}^2 q_{i_2} q_{i_3} + p_{i_2} p_{i_3} q_{i_1}^2 }$\\[0.8ex] \hline
        5 & $A_1A_2$ & $A_1A_2$ & $\fns{\begin{pmatrix} 1&1&0&0\\1&1&0&0 \end{pmatrix}}$ & $\fns{\begin{pmatrix} 2&0&0\\0&2&0 \\ 0&0&0  \end{pmatrix}}$ & $\frac{1}{2}$&$2\sum_{i_1\neq i_2}^I {p_{i_1} p_{i_2} q_{i_1} q_{i_2} }$\\[0.8ex] \hline
        6 & $A_1A_2$ & $A_1A_3$ &  $\fns{\begin{pmatrix} 1&1&0&0\\1&0&1&0 \end{pmatrix}}$ & $\fns{\begin{pmatrix} 1&1&0\\1&1&0 \\ 0&0&0 \end{pmatrix}}$ & $\frac{3}{4}$&$4\sum_{i_1\neq i_2\neq i_3}^I {p_{i_1} p_{i_2} q_{i_1} q_{i_3} }$\\[0.8ex] \hline
        7 & $A_1A_2$ & $A_3A_4$ & $\fns{\begin{pmatrix} 1&1&0&0\\0&0&1&1 \end{pmatrix}}$ & $\fns{\begin{pmatrix} 0&2&0\\2&0&0 \\ 0&0&0 \end{pmatrix}}$ & $1$&$\sum_{i_1\neq i_2\neq i_3\neq i_4}^I {p_{i_1} p_{i_2} q_{i_3} q_{i_4} }$\\[0.8ex] \hline
\end{longtable}

\section{Probabilities of identity states}

The first draw of size $K$ is taken from $\mathcal{A}^K_I$ with probability $p_i$ for drawing the $i$th object $A_i$. The second size-$K$ draw is taken from $\mathcal{A}^K_I$ with probability $q_i$ for drawing $A_i$. We find the probability of each identity state by finding the probability up to each of the two symmetries captured in our group action: the first, the $S_2$ component, leads us to find the probability of an unordered pair of unordered draws. The second, the $S_I$ component, leads us to find the probability that considers relabelings of the $I$ objects.

\subsection{Probability of an unordered pair of unordered draws}

Denote by $\g_1$ the first unordered draw, and let $X_1$ be an ordering of $\g_1$. Similarly, denote by $\g_2$ the second unordered draw, and let $X_2$ be an ordering of $\g_2$. The probabilities of the ordered draws $X_1$ and $X_2$ are
\begin{align*}
    \mathbbm{P}[X_1]=\prod_{i=1}^I p_i^{g_1^{(i)}},\\
    \mathbbm{P}[X_2]=\prod_{i=1}^I q_i^{g_2^{(i)}}.
\end{align*}
The probability of the unordered draw $\g_1$ is obtained by summing across all possible orderings, the permutations $X_1$ that produce vector $\g_1$. Each ordering has the same probability. The number of orderings is the multinomial coefficient ${K \choose g_1^{(1)},g_1^{(2)},\ldots,g_1^{(I)}} = K! \big/ ( \prod_{i=1}^I {g_1^{(i)}!}) $. We have 
\begin{align*}
    \mathbbm{P}[\g_1]  &= {K \choose g_1^{(1)},g_1^{(2)},\ldots,g_1^{(I)}}\prod_{i=1}^I p_i^{g_1^{(i)}}, \\   
    \mathbbm{P}[\g_2]  &= {K \choose g_2^{(1)},g_2^{(2)},\ldots,g_2^{(I)}}\prod_{i=1}^I q_i^{g_2^{(i)}}.        
\end{align*}

\vspace{-0.25cm}
\clearpage
\begin{landscape}
\begin{longtable}[!ht]{|l|c|c|c|c|c|l|}\caption{Enumeration of identity states in $\mathcal{C}^K_I$ for $K=2$, $I \geq 4$, 
each given by an associated pair of draws $[\hat{\g}]$ and a square matrix $M$; dissimilarities $\Db$ and probabilities are also shown.}\label{tab:4prob-combs-k3} \\
    \hline
        Case& $G_1$ & $G_2$ & $[\hat{\g}]$ & M & $\Db$ & Probability \\[0.8ex] \hline
        1 & $A_1A_1A_1$ & $A_1A_1A_1$ & ${\fns\begin{pmatrix} 3&0&0&0&0&0\\3&0&0&0&0&0 \end{pmatrix}}$ & ${\fns\begin{pmatrix} 5&0&0&0\\0&0&0&0 \\ 0&0&0&0 \\ 0&0&0&1 \end{pmatrix}}$ & $0$&$\sum_{i_1}^I {p_{i_1}^3 q_{i_1}^3 }$\\[0.8ex] \hline
        2 & $A_1A_1A_1$ & $A_1A_1A_2$ &  ${\fns\begin{pmatrix} 3&0&0&0&0&0\\2&1&0&0&0&0 \end{pmatrix}}$ & ${\fns\begin{pmatrix} 4&1&0&0\\0&0&0&0 \\ 0&0&0&0 \\ 0&0&1&0 \end{pmatrix}}$ & $\frac{1}{3}$&$3\sum_{i_1\neq i_2}^I {p_{i_1}^3 q_{i_1}^2 q_{i_2}+ p_{i_1}^2 p_{i_2}q_{i_1}^3}$\\[0.8ex] \hline
        3 & $A_1A_1A_1$ & $A_1A_2A_2$ & ${\fns\begin{pmatrix} 3&0&0&0&0&0\\1&2&0&0&0&0 \end{pmatrix}}$ & ${\fns\begin{pmatrix} 4&0&1&0\\0&0&0&0 \\ 0&0&0&0 \\ 0&1&0&0 \end{pmatrix}}$ & $\frac{2}{3}$&$3\sum_{i_1\neq i_2}^I {p_{i_1}^3 q_{i_1} q_{i_2}^2 + p_{i_1} p_{i_2}^2 q_{i_1}^3 }$\\[0.8ex] \hline
        4 & $A_1A_1A_1$ & $A_1A_2A_3$ &  ${\fns\begin{pmatrix} 3&0&0&0&0&0\\1&1&1&0&0&0 \end{pmatrix}}$ & ${\fns\begin{pmatrix} 3&2&0&0\\0&0&0&0 \\ 0&0&0&0 \\ 0&1&0&0 \end{pmatrix}}$ & $\frac{2}{3}$&$3\sum_{i_1\neq i_2\neq i_3}^I {p_{i_1}^3 q_{i_1} q_{i_2} q_{i_3} + p_{i_1} p_{i_2} p_{i_3} q_{i_1}^3 }$\\[0.8ex] \hline
        5 & $A_1A_1A_1$ & $A_2A_2A_2$ & ${\fns\begin{pmatrix} 3&0&0&0&0&0\\0&3&0&0&0&0 \end{pmatrix}}$ & ${\fns\begin{pmatrix} 4&0&0&1\\0&0&0&0 \\ 0&0&0&0 \\ 1&0&0&0 \end{pmatrix}}$ & $1$&$\sum_{i_1\neq i_2}^I {p_{i_1}^3 q_{i_2}^3}$\\[0.8ex] \hline
        6 & $A_1A_1A_1$ & $A_2A_2A_3$ & ${\fns\begin{pmatrix} 3&0&0&0&0&0\\0&2&1&0&0&0 \end{pmatrix}}$ & ${\fns\begin{pmatrix} 3&1&1&0\\0&0&0&0 \\ 0&0&0&0 \\ 1&0&0&0 \end{pmatrix}}$ & $1$&$3\sum_{i_1\neq i_2\neq i_3}^I {p_{i_1}^3 q_{i_2}^2 q_{i_3} + p_{i_2}^2 p_{i_3} q_{i_1}^3 }$\\[0.8ex] \hline
        7 & $A_1A_1A_1$ & $A_2A_3A_4$ &  ${\fns\begin{pmatrix} 3&0&0&0&0&0\\0&1&1&1&0&0 \end{pmatrix}}$ & ${\fns\begin{pmatrix} 2&3&0&0\\0&0&0&0 \\ 0&0&0&0 \\ 1&0&0&0 \end{pmatrix}}$ & $1$&$\sum_{i_1\neq i_2\neq i_3\neq i_4}^I {p_{i_1}^3 q_{i_2} q_{i_3} q_{i_4} + p_{i_2} p_{i_3} p_{i_4} q_{i_1}^3 }$\\[0.8ex] \hline
        8 & $A_1A_1A_2$ & $A_1A_1A_2$ & ${\fns\begin{pmatrix} 2&1&0&0&0&0\\2&1&0&0&0&0 \end{pmatrix}}$ & ${\fns\begin{pmatrix} 4&0&0&0\\0&1&0&0 \\ 0&0&1&0 \\ 0&0&0&0 \end{pmatrix}}$ & $\frac{4}{9}$&$9\sum_{i_1\neq i_2}^I {p_{i_1}^2 p_{i_2} q_{i_1}^2 q_{i_2} }$\\[0.8ex] \hline
        9 & $A_1A_1A_2$ & $A_1A_1A_3$ & ${\fns\begin{pmatrix} 2&1&0&0&0&0\\2&0&1&0&0&0 \end{pmatrix}}$ & ${\fns\begin{pmatrix} 3&1&0&0\\1&0&0&0 \\ 0&0&1&0 \\ 0&0&0&0 \end{pmatrix}}$ & $\frac{5}{9}$&$9\sum_{i_1\neq i_2\neq i_3}^I {p_{i_1}^2 p_{i_2} q_{i_1}^2 q_{i_3} }$\\[0.8ex] \hline
        10 & $A_1A_1A_2$ & $A_1A_2A_2$ & ${\fns\begin{pmatrix} 2&1&0&0&0&0\\1&2&0&0&0&0 \end{pmatrix}}$ & ${\fns\begin{pmatrix} 4&0&0&0\\0&0&1&0 \\ 0&1&0&0 \\ 0&0&0&0 \end{pmatrix}}$ & $\frac{5}{9}$&$9\sum_{i_1\neq i_2}^I {p_{i_1}^2 p_{i_2} q_{i_1} q_{i_2}^2 }$\\[0.8ex] \hline
        11 & $A_1A_1A_2$ & $A_1A_2A_3$ &  ${\fns\begin{pmatrix} 2&1&0&0&0&0\\1&1&1&0&0&0 \end{pmatrix}}$ & ${\fns\begin{pmatrix} 3&1&0&0\\0&1&0&0 \\ 0&1&0&0 \\ 0&0&0&0 \end{pmatrix}}$ & $\frac{2}{3}$&$18\sum_{i_1\neq i_2\neq i_3}^I {p_{i_1}^2 p_{i_2} q_{i_1} q_{i_2} q_{i_3} + p_{i_1} p_{i_2} p_{i_3} q_{i_1}^2 q_{i_2} }$\\[0.8ex] \hline
        12 & $A_1A_1A_2$ & $A_1A_3A_3$ & ${\fns\begin{pmatrix} 2&1&0&0&0&0\\1&0&2&0&0&0 \end{pmatrix}}$ & ${\fns\begin{pmatrix} 3&0&1&0\\1&0&0&0 \\ 0&1&0&0 \\ 0&0&0&0 \end{pmatrix}}$ & $\frac{7}{9}$&$9\sum_{i_1\neq i_2\neq i_3}^I {p_{i_1}^2 p_{i_2} q_{i_1} q_{i_3}^2 + p_{i_1} p_{i_3}^2 q_{i_1}^2 q_{i_2}}$\\[0.8ex] \hline
        13 & $A_1A_1A_2$ & $A_1A_3A_4$ &  ${\fns\begin{pmatrix} 2&1&0&0&0&0\\1&0&1&1&0&0 \end{pmatrix}}$ & ${\fns\begin{pmatrix} 2&2&0&0\\1&0&0&0 \\ 0&1&0&0 \\ 0&0&0&0 \end{pmatrix}}$ & $\frac{7}{9}$&$9\sum_{i_1\neq i_2\neq i_3\neq i_4}^I {p_{i_1}^2 p_{i_2} q_{i_2} q_{i_3} q_{i_4} + p_{i_2} p_{i_3} p_{i_4} q_{i_1}^2 q_{i_2}^2 }$\\[0.8ex] \hline
        14 & $A_1A_1A_2$ & $A_2A_3A_3$ & ${\fns\begin{pmatrix} 2&1&0&0&0&0\\0&1&2&0&0&0 \end{pmatrix}}$ & ${\fns\begin{pmatrix} 3&0&1&0\\0&1&0&0 \\ 1&0&0&0 \\ 0&0&0&0 \end{pmatrix}}$ & $\frac{8}{9}$&$9\sum_{i_1\neq i_2\neq i_3}^I {p_{i_1}^2 p_{i_2} q_{i_2} q_{i_3}^2}$\\[0.8ex] \hline
        15 & $A_1A_1A_2$ & $A_2A_3A_4$ & ${\fns\begin{pmatrix} 2&1&0&0&0&0\\0&1&1&1&0&0 \end{pmatrix}}$ & ${\fns\begin{pmatrix} 2&2&0&0\\0&1&0&0 \\ 1&0&0&0 \\ 0&0&0&0 \end{pmatrix}}$ & $\frac{8}{9}$&$9\sum_{i_1\neq i_2\neq i_3\neq i_4}^I {p_{i_1}^2 p_{i_2} q_{i_2} q_{i_3} q_{i_4} + p_{i_2} p_{i_3} p_{i_4} q_{i_1}^2 q_{i_2}}$\\[0.8ex] \hline
        16 & $A_1A_1A_2$ & $A_3A_3A_4$ &  ${\fns\begin{pmatrix} 2&1&0&0&0&0\\0&0&2&1&0&0 \end{pmatrix}}$ & ${\fns\begin{pmatrix} 2&1&1&0\\1&0&0&0 \\ 1&0&0&0 \\ 0&0&0&0 \end{pmatrix}}$  & $1$&$9\sum_{i_1\neq i_2\neq i_3\neq i_4}^I {p_{i_1}^2 p_{i_2} q_{i_3}^2 q_{i_4}}$\\[0.8ex] \hline
        17 & $A_1A_1A_2$ & $A_3A_4A_5$ & ${\fns\begin{pmatrix} 2&1&0&0&0&0\\0&0&1&1&1&0 \end{pmatrix}}$ & ${\fns\begin{pmatrix} 1&3&0&0\\1&0&0&0 \\ 1&0&0&0 \\ 0&0&0&0 \end{pmatrix}}$ & $1$&$3\sum_{i_1\neq i_2\neq i_3\neq i_4\neq i_5}^I {p_{i_1}^2 p_{i_2} q_{i_3} q_{i_4} q_{i_5} + p_{i_3} p_{i_4} p_{i_5} q_{i_1}^2 q_{i_2}}$\\[0.8ex] \hline
        18 & $A_1A_2A_3$ & $A_1A_2A_3$ & ${\fns\begin{pmatrix} 1&1&1&0&0&0\\1&1&1&0&0&0 \end{pmatrix}}$ & ${\fns\begin{pmatrix} 3&0&0&0\\0&3&0&0 \\ 0&0&0&0 \\ 0&0&0&0 \end{pmatrix}}$ & $\frac{2}{3}$&$6\sum_{i_1\neq i_2\neq i_3}^I {p_{i_1} p_{i_2} p_{i_3} q_{i_1} q_{i_2} q_{i_3} }$\\[0.8ex] \hline
        19 & $A_1A_2A_3$ & $A_1A_2A_4$ &  ${\fns\begin{pmatrix} 1&1&1&0&0&0\\1&1&0&1&0&0 \end{pmatrix}}$ & ${\fns\begin{pmatrix} 2&1&0&0\\1&2&0&0 \\ 0&0&0&0 \\ 0&0&0&0 \end{pmatrix}}$  & $\frac{7}{9}$&$18\sum_{i_1\neq i_2\neq i_3\neq i_4}^I {p_{i_1} p_{i_2} p_{i_3} q_{i_1} q_{i_2} q_{i_4}}$\\[0.8ex] \hline
        20 & $A_1A_2A_3$ & $A_1A_4A_5$ & ${\fns\begin{pmatrix} 1&1&1&0&0&0\\1&0&0&1&1&0 \end{pmatrix}}$ & ${\fns\begin{pmatrix} 1&2&0&0\\2&1&0&0 \\ 0&0&0&0 \\ 0&0&0&0 \end{pmatrix}}$ & $\frac{8}{9}$&$9\sum_{i_1\neq i_2\neq i_3\neq i_4\neq i_5}^I {p_{i_1} p_{i_2} p_{i_3} q_{i_1} p_{i_4} p_{i_5}}$\\[0.8ex] \hline
        21 & $A_1A_2A_3$ & $A_4A_5A_6$ & ${\fns\begin{pmatrix} 1&1&1&0&0&0\\0&0&0&1&1&1 \end{pmatrix}}$ & ${\fns\begin{pmatrix} 0&3&0&0\\3&0&0&0 \\ 0&0&0&0 \\ 0&0&0&0 \end{pmatrix}}$ & $1$&$\sum_{i_1\neq i_2\neq i_3\neq i_4\neq i_5\neq i_6}^I {p_{i_1} p_{i_2} p_{i_3} q_{i_4} q_{i_5} q_{i_6} }$\\[0.8ex] \hline
\end{longtable}
\end{landscape}

\newpage
\begin{longtable}[tb]{|r|c|c|c|c|c|c|}
\caption{$\big|\mathcal{C}_I^K \big|$, the number of identity states for unordered pairs of unordered draws of size $K$ from a set of $I$ objects, for small values of $K$. For $I \geq 2K$, $\big|\mathcal{C}_I^K \big| = \big|\mathcal{C}_{2K}^K \big|$.}\label{tab:5}
\\\hline
    & \multicolumn{6}{|c|}{$K$} \\ \hline
$I$ & 1 & 2 &  3 &  4 &   5 &   6 \\ \hline
1   & 1 & 1 &  1 &  1 &   1 &   1 \\ \hline
2   & 2 & 4 &  6 &  9 &  12 &  16 \\ \hline
3   &   & 6 & 13 & 26 &  46 &  79 \\ \hline
4   &   &   & 18 & 45 &  96 & 200 \\ \hline
5   &   &   & 20 & 57 & 140 & 333 \\ \hline
6   &   &   & 21 & 63 & 169 & 440 \\ \hline
7   &   &   &    & 65 & 183 & 506 \\ \hline
8   &   &   &    & 66 & 189 & 541 \\ \hline
9   &   &   &    &    & 191 & 556 \\ \hline
10  &   &   &    &    & 192 & 562 \\ \hline
11  &   &   &    &    &     & 564 \\ \hline
12  &   &   &    &    &     & 565 \\ \hline
\end{longtable}

Therefore, the probability of obtaining an \textit{ordered} pair consisting of draws $\g_1$ and $\g_2$ is given by the product $\mathbbm{P}[\hat{\g}] = \mathbbm{P}[\g_1] \, \mathbbm{P}[\g_2]$, or 
\begin{align}
    \mathbbm{P}[\hat{\g}] &={K \choose g_1^{(1)},g_1^{(1)},\ldots,g_1^{(I)}} {K \choose g_2^{(1)},g_2^{(2)},\ldots,g_2^{(I)}}\prod_{i=1}^I p_i^{g_1^{(i)}} q_i^{g_2^{(i)}}.
\end{align}

To obtain the probability of an \textit{unordered} pair, we must consider the case in which $\g_1$ is drawn from probability distribution $(q_1,q_2, \ldots, q_I)$ and $\g_2$ is drawn from  $(p_1,p_2,\ldots,p_I)$. If $\g_1=\g_2$, then this case is identical to the previous one. If $\g_1 \neq \g_2$, then we combine two cases into a single event $[\hat{\g}]_\sim$ to compute the probability of the unordered pair of unordered draws:
\begin{align}
     \mathbbm{P}\big[[\hat{\g}]_\sim] &= \frac{1}{1+ \mathbbm{1}_{\g_1=\g_2}} \bigg[\mathbbm{P}\Big[{\g_1 \choose \g_2}\Big] +  \mathbbm{P}\Big[{\g_2 \choose \g_1}\Big]\bigg]
     \nonumber \\
     &= \frac{1}{1+ \mathbbm{1}_{\g_1=\g_2}}{K \choose g_1^{(1)},g_1^{(2)},\ldots,g_1^{(I)}} {K \choose g_2^{(1)},g_2^{(2)},\ldots,g_2^{(I)}}\bigg(\prod_{i=1}^I p_i^{g_1^{(i)}} q_i^{g_2^{(i)}} + \prod_{i=1}^I p_i^{g_2^{(i)}} q_i^{g_1^{(i)}} \bigg).
\label{eq:row}
\end{align}
Here, we sum across the two possible orderings of the draws, dividing by $1+ \mathbbm{1}_{\g_1=\g_2}$ to account for the case in which the two draws are identical. 

\subsection{Relabeling the objects}

To find the probability of an identity state $[\hat{\g}]\in\mathcal{C}^K_I$, we must sum probabilities of all compositions of a row arrangement and a labeling of the objects that give rise to the same identity state. 

Let $N(\hat{\g})$ be the number of distinct objects in $\hat{\g}$, or equivalently, the number of nonzero columns in its associated 2-row matrix. Without loss of generality, we can assume that the nonzero columns are the first $N(\hat{\g})$ columns of $\hat{\g}$, as we can always permute the columns to obtain such a matrix. A relabeling of the objects $A_i$ corresponds to a permutation  of their frequencies $(p_i,q_i)$. Therefore, to find the probability of $[\hat{\g}]$, we sum probabilities for the two row arrangements across all distinct reassignments of $(p_i,q_i)$. We have already addressed the possibility of two distinct row arrangements in eq.~\ref{eq:row}. We therefore obtain
\begin{align}
\label{eq:diff-pop-case-prob}
\mathbb{P}\big[ [\hat{\g}] \big]= C_N(\hat{\g}) \, {K \choose g_1^{(1)},g_1^{(2)},\ldots,g_1^{(N(\hat{\g}))}} {K \choose g_2^{(1)},g_2^{(2)},\ldots,g_2^{(N(\hat{\g}))}} \!\!\!\!\!\sum_{\;\;\;\;\; i_1 \neq i_2 \neq \ldots \neq i_{N(\hat{\g})}}
\frac{ \prod_{j=1}^{N(\hat{\g})} p_{i_j}^{g_1^{(j)}} q_{i_j}^{g_2^{(j)}} + \prod_{j=1}^{N(\hat{\g})} p_{i_j}^{g_2^{(j)}} q_{i_j}^{g_1^{(j)}} }{1+ \mathbbm{1}_{\g_1=\g_2}},
\end{align}
where $\sum_{i_1 \neq i_2 \neq \ldots \neq i_k}$ is a sum over all possible vectors $(i_1, i_2, \ldots, i_k)$ that represent permutations of $\{1,2,\ldots, k\}$. 

The coefficient $C_N(\hat{\g})$ accounts for distinct vectors  $(i_1, i_2, \ldots, i_{N(\hat{\g})})$ that repeat the probability in the summand. To compute this coefficient, we write $\g_1 \sim \g_2$ if there exists $\sigma \in S_{N(\hat{\g})}$ that maps $\g_1 \in \mathcal{G}_I^K$ to $\g_2 \in \mathcal{G}_I^K$. In other words, $\g_1 \sim \g_2$ if and only if $\g_1$ and $\g_2$ are constructed from the same unordered partition of $K$. In terms of a matrix with rows $\g_1$ and $\g_2$, $\g_1 \sim \g_2$ if and only if the two rows are the same up to a permutation of elements. Let $r(\g)$ be a vector-valued function of length $K+1$ indexed from $0$ to $K$, where $r_k$, $k=0,1,\ldots,K$, is the number of entries in the unordered size-$K$ draw $\g$ that equal $k$. Then $\g_1 \sim \g_2$ if and only if $r(\g_1)=r(\g_2)$.

Computation of $C_N(\hat{\g})$ uses the size of a stabilizer subgroup. For identical $\g_1, \g_2$, a permutation $\sigma \in S_{N(\hat{\g})}$ tabulates the same probability in the summand of eq.~\ref{eq:diff-pop-case-prob} as the identity permutation if and only if $\sigma(\hat{\g})=\hat{\g}$. For distinct $\g_1, \g_2$, $\sigma \in S_{N(\hat{\g})}$ tabulates the same probability as the identity if $\sigma(\hat{\g})=\hat{\g}$. If $\hat{\g}_1 \neq \hat{\g}_2$ but $\hat{\g}_1 \sim \hat{\g}_2$, however, then each product $\prod_{j=1}^{N(\hat{\g})} p_{i_j}^{g_1^{(j)}} q_{i_j}^{g_2^{(j)}}$ reappears as $\prod_{j=1}^{N(\hat{\g})} p_{i_j}^{g_2^{(j)}} q_{i_j}^{g_1^{(j)}}$ if $\sigma{(\g_1)=\g_2}$, and each product $\prod_{j=1}^{N(\hat{\g})} p_{i_j}^{g_2^{(j)}} q_{i_j}^{g_1^{(j)}}$ reappears as  $\prod_{j=1}^{N(\hat{\g})} p_{i_j}^{g_1^{(j)}} q_{i_j}^{g_2^{(j)}}$ if $\sigma{(\g_2)=\g_1}$. In this case, for each element in the stabilizer $\text{Stab}_{S_{N(\hat{\g})}} [\hat{\g}]$, \emph{two} permutations recover the term  $\prod_{j=1}^{N(\hat{\g})} p_{i_j}^{g_1^{(j)}} q_{i_j}^{g_2^{(j)}}$. The first is the element itself. The second is the element composed with the permutation $\sigma_*$ that has $\sigma_*(\g_2)=\g_1$. The latter permutation recovers the term from the second product in eq.~\ref{eq:diff-pop-case-prob}, with the $p$ and $q$ exchanged. Similarly, two permutations recover $\prod_{j=1}^{N(\hat{\g})} p_{i_j}^{g_1^{(j)}} q_{i_j}^{g_2^{(j)}}$. Therefore, 
\begin{align}
\label{eq:diff-pop-case-prob-N}
    C_N(\hat{\g}) & = 
    \begin{cases}
        \frac{1}{\big| \text{Stab}_{S_{N(\hat{\g})}} [\hat{\g}] \big|}, & \g_1 = \g_2, \\   \frac{1}{[1 + \mathbbm{1}_{\g_1 \sim  \g_2}]\big| \text{Stab}_{S_{N(\hat{\g})}} [\hat{\g}] \big|}, & \g_1 \neq \g_2.
\end{cases}
\end{align}
We now compute the size of the stabilizer subgroup via the matrix representation.

\subsection{Computing the stabilizer}

An element of $S_{N(\hat{\g})}$ acts on $\hat{\g} \in \mathcal{G}_I^K \times \mathcal{G}_I^K$ by permuting its nonzero columns---the first $N(\hat{\g})$ columns. To find $\big| \text{Stab}_{S_{N(\hat{\g})}} [\hat{\g}] \big|$, we count all permutations of the nonzero columns of $\hat{\g}$ that yield the same matrix.

Place the nonzero columns of matrix $\hat{\g}$ into $L$ equivalence classes $c_1, c_2, \ldots, c_L$, where columns are placed in the same class if and only if they are equal. Denote by $\{c_\ell\}$ the set of all nonzero columns in $\hat{\g}$ that equal $c_\ell$, with $\sum_{\ell=1}^L |\{c_\ell\}| = N(\hat{\g})$. The size of the stabilizer $\big| \text{Stab}_{S_{N(\hat{\g})}} [\hat{\g}] \big|$ in $C_N([\hat{\g}])$ is 
\begin{equation}\label{eq:denom-matrix-form}
    \big| \text{Stab}_{S_{N(\hat{\g})}} [\hat{\g}] \big| = \prod_{\ell=1}^L \lvert \{c_\ell\} \rvert!.
\end{equation}
The product counts ways to rearrange nonzero columns in the matrix $\hat{\g}$, amounting to relabelings of the objects, while retaining the order of the two draws.

\medskip

\subsection{Final probability expression}
\label{sec:finalprob}

Note that if $\mathbbm{1}_{\g_1=\g_2}=1$, then $\mathbbm{1}_{\g_1 \sim  \g_2}=1$. Therefore,
\begin{align}
    (1+\mathbbm{1}_{\g_1=\g_2})[1+ (\mathbbm{1}_{\g_1 \sim \g_2}-\mathbbm{1}_{\g_1=\g_2})] = 1+ \mathbbm{1}_{\g_1 \sim \g_2}.
\end{align}
Substituting our result from eq.~\ref{eq:denom-matrix-form} into our probability expression from eq.~\ref{eq:diff-pop-case-prob} and accommodating the two cases of $C_N([\hat{\g}]$ by an indicator function, we obtain the probability of an identity state.
\begin{theorem}
\label{thm:prob}
Suppose two sets of $K$ unordered items are drawn from a set of $I$ objects with replacement, one according to probability distribution $\mathbf{p}$ and the other according to probability distribution $\mathbf{q}$. The probability that the unordered pair of $K$ unordered items has identity state $[\hat{\g}]$ is 
\begin{align}
    \mathbb{P}\big[[\hat{\g}]\big] &= \frac{1}{(1+ \mathbbm{1}_{\g_1 \sim \g_2})(\prod_{\ell=1}^L \lvert \{c_\ell\} \rvert!)}{K \choose g_1^{(1)},g_1^{(2)},\ldots,g_1^{(I)}} {K \choose g_2^{(1)},g_2^{(2)},\ldots,g_2^{(I)}} \!\!\!\!\!\ \\
    & \qquad \times \sum_{\;\;\;\;\; i_1 \neq i_2 \neq \ldots \neq i_{N(\hat{\g})}} \bigg( \prod_{j=1}^{N(\hat{\g})} p_{i_j}^{g_1^{(j)}} q_{i_j}^{g_2^{(j)}} + \prod_{i=1}^{N(\hat{\g})} p_{i_j}^{g_2^{(j)}} q_{i_j}^{g_1^{(j)}} \bigg).
\end{align}
\end{theorem}

In the case that $\mathbf{p}=\mathbf{q}$ and the two unordered draws are sampled from the same probability distribution, the probability can be simplified.

\begin{corollary}
Suppose two sets of $K$ unordered items are drawn from a set of $I$ objects with replacement, both according to probability distribution $\mathbf{p}$. The probability that the unordered pair of $K$ unordered items has identity state $[\hat{\g}]$ is 
\begin{align}
    \mathbb{P}[\hat{\g}])= \frac{2}{(1+ \mathbbm{1}_{\g_1 \sim \g_2})(\prod_{\ell=1}^L \lvert \{c_\ell\} \rvert!)}{K \choose g_1^{(1)},g_1^{(2)},\ldots,g_1^{(I)}} {K \choose g_2^{(1)},g_2^{(2)},\ldots,g_2^{(I)}}\!\!\!\!\!\sum_{\;\;\;\;\; i_1 \neq i_2 \neq \ldots \neq i_{N(\hat{\g})}}\prod_{j=1}^{N(\hat{\g})} p_{i_j}^{g_1^{(j)}+g_2^{(j)}}.
\end{align}
\end{corollary}
We are now able to compute the probability of each of the $\mathcal{C}^K_{2K}$ identity states. For the cases of $K=2$ and $K=3$, these probabilities appear in Tables \ref{tab:3prob-combs-k2} and \ref{tab:4prob-combs-k3}. We continue to use the notation $\!\sum_{i_1\neq i_2}^I \!\!a_{i_1 i_2}= \sum_{i_1=1}^I \sum_{i_2=1, {i_2 \neq i_1}}^I \!\!a_{i_1 i_2}$, $\sum_{i_1\neq i_2 \neq i_3}^I \!\!a_{i_1 i_2 i_3}= \sum_{i_1=1}^I \sum_{i_2=1, {i_2 \neq i_1}}^I \sum_{i_3=1, {{i_3 \neq i_1}, i_3 \neq i_2}}^I \!\!a_{i_1 i_2 i_3}$, and so on.

\section{Expected dissimilarity value}

Let $\E[\Db(\mathbf{p},\mathbf{q})]$ be the expected dissimilarity between two random unordered draws with replacement as a function of our drawing probability vectors $\mathbf{p} =(p_1,p_2,\ldots,p_I)$ and $\mathbf{q} =(q_1,q_2,\ldots,q_I)$. We can compute $\E[\Db(\mathbf{p},\mathbf{q})]$ for a given $K$ by taking the dissimilarity of each identity state and its corresponding probability, as computed in Section \ref{sec:finalprob}:
\begin{align}
    \E[\Db(\mathbf{p},\mathbf{q})] = \sum_{[\hat{\g}]\in\mathcal{C}^K_I}\Db(\hat{\g}) \, \mathbb{P}\big[[\hat{\g}]\big].
\end{align}
For the $K=2$ case, the computation is equivalent to to taking the dot product of the $\Db$ and Probability columns of Table \ref{tab:3prob-combs-k2}, and for $K=3$, it is equivalent to taking the corresponding dot product in Table \ref{tab:4prob-combs-k3}. The resulting polynomials are reduced by noting $p_1+p_2+\ldots+p_I-1=0$ and $q_1+q_2+\ldots +q_I-1=0$. We obtain a general theorem.

\begin{theorem}
\label{thm:dissimilarity} For each choice of $K \geq 2$ and $I \geq 2$ and probability distributions $\mathbf{p}, \mathbf{q}$,
    \begin{align}
        \E[\Db(\mathbf{p},\mathbf{q})] = 1- \langle \mathbf{p}, \mathbf{q} \rangle.
    \end{align}
\end{theorem}
\begin{proof}
    Let $G_1,G_2$ be independent random variables in $\mathcal{G}_K^I$ corresponding to our two unordered draws.
    For two instances $g_1$ of $G_1$ and $g_2$ of $G_2$, eq.~\ref{eq:dg1g2} gives 
    \begin{align*}
        \Db(g_1,g_2)=1-\frac{1}{K^2}\sum_{i=1}^K \sum_{j=1}^K \mathbbm{1}_{g_1^{(i)} = g_2^{(j)}}.
    \end{align*}
    In other words, using $S$ and $T$ for random variables corresponding to randomly selected indices in $[1,K]$,
    \begin{align*}
        \Db(g_1,g_2)= \E[\mathbbm{1}_{g_1^{(S)} \neq g_2^{(T)}}].
    \end{align*}
    
    By the law of total expectation,
    \begin{align}
    \label{eq:totalexp}
        \E[\Db(\mathbf{p}, \mathbf{q})] &= \E_{G_1,G_2}[ \Db(g_1,g_2)] = \E_{G_1,G_2} \big[\E_{S,T}[\mathbbm{1}_{G_1^S \neq G_2^T}]  \big] = \E_{G_1,G_2}\big[ \mathbb{P}[G_1^S \neq G_2^T] \big]. 
    \end{align}
    We can compute this latter expectation by summing across all outcomes for the independent random variables $G_1$, $G_2$, $S$, and $T$: 
    \begin{align}
    \label{eq:egg}
         \E_{G_1,G_2}\big[ \mathbb{P}[G_1^S \neq G_2^T] \big]
        &=  \sum_{(G_1,G_2) \in \mathcal{G}^K_I\times \mathcal{G}^K_I} \mathbb{P}[G_1=g_1] \, \mathbb{P}[G_2=g_2] \sum_{1\leq s, t \leq K}  \mathbb{P}[S\!=\!s] 
        \, \mathbb{P}[T\!=\!t] \, \mathbbm{1}_{g_1^{(s)} \neq g_2^{(t)}}.
    \end{align}
For fixed $s$ and $t$, 
    \begin{align}
    \label{eq:fixedst}
\sum_{(G_1,G_2) \in \mathcal{G}^K_I\times \mathcal{G}^K_I}  \mathbb{P}[G_1=g_1] \, \mathbb{P}[G_2=g_2] \, \mathbbm{1}_{g_1^{(s)} \neq g_2^{(t)}} =  1- \langle \mathbf{p}, \mathbf{q} \rangle, 
    \end{align}
encoding the fact that the probability that a random draw of one element from population 1 and one element from population 2 represent the same object is $\langle \mathbf{p},\mathbf{q} \rangle$.
    Therefore, applying eqs.~\ref{eq:egg} and \ref{eq:fixedst} in eq.~\ref{eq:totalexp},
    \begin{align}
        \E[\Db(\mathbf{p},\mathbf{q})] &= \sum_{1\leq s, t \leq K} \sum_{(G_1,G_2) \in \mathcal{G}^K_I\times \mathcal{G}^K_I}  \mathbb{P}[G_1=g_1] \, \mathbb{P}[G_2=g_2] \, \mathbb{P}[S\!=\!s] \, \mathbb{P}[T\!=\!t] \, \mathbbm{1}_{g_1^{(s)} \neq g_2^{(t)}} \nonumber \\
        &= \sum_{1\leq s, t \leq K} \mathbb{P}[S\!=\!s] \, \mathbb{P}[T\!=\!t] 
        \Big[\sum_{(G_1,G_2) \in \mathcal{G}^K_I\times \mathcal{G}^K_I}  \mathbb{P}[G_1=g_1] \, \mathbb{P}[G_2=g_2] \, \mathbbm{1}_{g_1^{(s)} \neq g_2^{(t)}}\Big] \nonumber \\
        &=  \sum_{1\leq s, t \leq K} \mathbb{P}[S\!=\!s] \, \mathbb{P}[T\!=\!t] \, (1- \langle \mathbf{p}, \mathbf{q} \rangle) \nonumber \\
        &= \sum_{1\leq s, t \leq K} \frac{1}{K^2} (1- \langle \mathbf{p}, \mathbf{q} \rangle) \nonumber \\
        &= 1- \langle \mathbf{p}, \mathbf{q} \rangle.  \nonumber
    \end{align}
\end{proof}

We can immediately discern the conditions under which the expected dissimilarity between draws from two different populations exceeds that of two draws from the same population.
\begin{corollary}
\label{coro:pppq} 
For each choice of $K \geq 2$ and $I \geq 2$ and probability distributions $\mathbf{p}, \mathbf{q}$,
    $\E[\Db(\mathbf{p},\mathbf{p})] \leq \E[\Db(\mathbf{p},\mathbf{q})]$ if and only if $\langle \mathbf{p}, \mathbf{q} \rangle \leq \langle \mathbf{p}, \mathbf{p} \rangle$.
\end{corollary}
\begin{proof}
We apply Theorem \ref{thm:dissimilarity} twice, finding that  $\E[\Db(\mathbf{p},\mathbf{p})] \leq \E[\Db(\mathbf{p},\mathbf{q})]$ is equivalent to
$ 1 - \langle \mathbf{p}, \mathbf{p} \rangle \leq 1 - \langle \mathbf{p}, \mathbf{q} \rangle $.
\end{proof}
The corollary clarifies that there exist probability distributions for which the expected dissimilarity between draws from the same probability distribution exceeds that of draws from different probability distributions. For example, for $\mathbf{p} = (0.8, 0.2, 0, \ldots, 0)$ and $\mathbf{q} = (0.9, 0.1, 0, \ldots, 0)$, 
\begin{align*}
    \E[\Db(\mathbf{p},\mathbf{p})] = 1- 0.68 = 0.32 \geq 0.26 = \E[\Db(\mathbf{p},\mathbf{q})].
\end{align*}
However, we do find that the expected dissimilarity for draws from distinct distributions is greater than or equal to that of at least one of the two constituent distributions. In particular, we have the following corollary.
\begin{theorem}
\label{thm:avg}
    For each choice of $K\geq 2$ and $I \geq 2$ and probability distributions $\mathbf{p}$, $\mathbf{q}$,
       \begin{align*}
           \frac{1}{2}\big(\E[\Db(\mathbf{p},\mathbf{p})] + \E[\Db(\mathbf{q},\mathbf{q})]\big) \leq \E[\Db(\mathbf{p},\mathbf{q})],
       \end{align*}
with equality if and only if $\mathbf{p}=\mathbf{q}$.
\end{theorem}
\begin{proof}
We apply Theorem \ref{thm:dissimilarity} three times, finding
    \begin{align*}
        \frac{1}{2}\big(\E[\Db(\mathbf{p},\mathbf{p})] + \E[\Db(\mathbf{q},\mathbf{q})]\big) - \E[\Db(\mathbf{p},\mathbf{q})] &= -\frac{1}{2} \langle \mathbf{p}, \mathbf{p} \rangle - \frac{1}{2} \langle \mathbf{q}, \mathbf{q} \rangle + \langle \mathbf{p}, \mathbf{q} \rangle\\
        &= - \frac{1}{2} \langle \mathbf{p}-\mathbf{q}, \mathbf{p}-\mathbf{q} \rangle\\
        &\leq 0.
\end{align*}
Equality holds if and only if $\mathbf{p}=\mathbf{q}$.
\end{proof}
As a result of the theorem, $\E[\Db(\mathbf{p},\mathbf{q})]$ is always greater than or equal to at least one of the two quantities $\E[\Db(\mathbf{p},\mathbf{p})]$, $\E[\Db(\mathbf{q},\mathbf{q})]$. Theorems \ref{thm:dissimilarity} and \ref{thm:avg} and Corollary \ref{coro:pppq} generalize corresponding results obtained by \cite{Liu2023} in the $K=2$ case, showing that as a measure of genetic differentiation between a pair of populations, $\Db$ does not depend on the ploidy $K$.

\section{Discussion}

We have examined the problem of measuring dissimilarity between two random, unordered size-$K$ draws with replacement from a set of $I$ objects. The problem considers draws from different probability vectors, relying on a dissimilarity measure that assesses their identity by considering every pairing of elements, one from one draw and one from the other. Via quotients by group actions, we have enumerated the ``identity states'' that describe configurations of identity and non-identity among the objects drawn (Section \ref{sec:enumeration}). We have also calculated the probability of each identity state (Theorem~\ref{thm:prob}). For the particular dissimilarity measure we have considered, we have shown that although the identity state probabilities are, in many cases, relatively complicated expressions, the expectation of the dissimilarity is a simple function of the starting probability vectors (Theorem~\ref{thm:dissimilarity}). This result generalizes an earlier result for $K=2$~\citep[eq.~22]{Liu2023}. 

We have shown that as in the case of $K=2$, for $K > 2$, it is possible for two random draws from the same probability vector to possess greater expected dissimilarity than do two draws from different vectors (Corollary~\ref{coro:pppq}). Nevertheless, the expected dissimilarity for different probability vectors $\mathbf{p}$ and $\mathbf{q}$ is bounded below by one of the expected dissimilarities taking both draws from the same vector, either  $\mathbf{p}$ or $\mathbf{q}$; in particular, we have $\frac{1}{2}(\E[\Db(\mathbf{p},\mathbf{p})] + \E[\Db(\mathbf{q},\mathbf{q})]) \leq \E[\Db(\mathbf{p},\mathbf{q})]$ (Theorem~\ref{thm:avg}). 

In the population-genetic motivation for the problem, $\E[\Db(\mathbf{p},\mathbf{q})]$ measures the dissimilarity between two populations at a genetic locus given vectors of the population allele frequencies, and it provides a measure of intrapopulation genetic variation in the case of $\mathbf{p}=\mathbf{q}$. Our expected dissimilarity (Theorem~\ref{thm:dissimilarity}) indicates that a dissimilarity designed for comparing polyploid genotypes ($K>2$) produces an expectation as simple as that obtained for one comparing diploid genotypes ($K=2$). The result contributes to the development of population-genetic statistics specifically for polyploids~\citep{RosenbergAndCalabrese04, ObbardEtAl06, FalushEtAl07, MeirmansLiu18, MeirmansEtAl18, YangEtAl21}.

As part of our derivations, we have enumerated a class of identity states for an unordered pair of $K$ unordered samples with replacement from a set of $I$ objects. In population genetics, identity states are useful for understanding diverse features of the transmission of alleles, relatedness of pairs of individuals, and properties of genetic identity; they have been most frequently considered for pairs of diploid individuals, corresponding to the case of $K=2$, with $I \geq 4$~\citep{Jacquard74, Thompson1974, Lange02}. Familiar concepts include the 15 identity states possible for an ordered pair of ordered diploid genotypes, and the 9 ``condensed'' identity states possible for an ordered pair of unordered diploid genotypes. These 9 states are used for the scenario in which alleles paternally and maternally transmitted from parent to offspring are not distinguished in a diploid offspring individual. If the two individuals in a pair are unordered, then the 9 condensed identity states collapse to our 7 states for unordered pairs of unordered diploid genotypes (Table~\ref{tab:3prob-combs-k2}). Generalizations of identity state concepts beyond two size-2 draws have focused on larger sets of draws, as would be relevant for multiple diploid individuals~\citep{Thompson1974, Karigl82}, rather than on draws of larger size, as are relevant for pairs of polyploid individuals. Our results on polyploid identity states provide a direction for generalization of classic aspects of genetic identity configurations.

Although the motivating scenario is from population genetics, we have described our results in the more general context of random, unordered draws. For example, consider two infinite stacks of playing cards, where each stack has an associated probability vector for drawing a card of type $A_i$. Each of two players draws $K$ cards to form a hand, with player 1 drawing from one stack and player 2 from the other. Compute the dissimilarity between the two hands via $\Db$. Theorem \ref{thm:dissimilarity} finds that the expectation of this dissimilarity is the same as if only one card was drawn from each stack. Suppose the game is structured so that player 1 seeks to maximize the probability of a card match with player 2 by assembling a deck of cards whose card proportions are specified in advance of the game. Corollary \ref{coro:pppq} finds that accomplishing this objective does not necessarily require matching the card probabilities in stack 1 to those of stack 2.

Several open questions remain.  The algorithm for enumerating the identity states (Section \ref{sec:enumeration}) is not very instructive; in effect, it amounts to reducing a space of nonnegative-integer matrices according to our group actions. We also have not determined a closed-form expression, generating function, or asymptotics for $|\mathcal{C}^K_I|$. Additionally, although we did find that the expected dissimilarity can be greater for two draws with identical probability distributions than for those with distinct distributions (Corollary \ref{coro:pppq}), we conjecture that this scenario is unlikely across potential pairs of probability vectors. In the $K=2$ case, the probability that $\E[\Db(\mathbf{p},\mathbf{p})]$ exceeds $\E[\Db(\mathbf{p},\mathbf{q})]$ decreases with increasing $I$, the number of nonzero-frequency objects~\citep[Figure 5]{Liu2023}. This probability and its limit can potentially also be analyzed for $K>2$.

\vskip .3cm
\noindent {\bf Acknowledgments.} We acknowledge support from National Institutes of Health grant R01 HG005855.

{\small
\bibliographystyle{abbrvnat}
\bibliography{references}
}

\end{document}